\documentclass[11pt]{amsart}
\usepackage{amssymb,amsmath}
\usepackage{graphicx}
\usepackage{bm,bbm}
\usepackage{color}
\usepackage{tikz}
\usepackage{ifthen}
\usetikzlibrary{patterns}
\usetikzlibrary{snakes}
\usepackage{algorithm}
\usepackage{setspace}
\usepackage[noend]{algpseudocode}
\usepackage{pgfplots}
\usepackage{mathtools}
\usepackage[hidelinks]{hyperref}
\newtheorem{theorem}{Theorem}

\theoremstyle{definition}

\theoremstyle{lemma}

\theoremstyle{remark}
\newtheorem{remark}[theorem]{Remark}

\numberwithin{theorem}{section}
\numberwithin{equation}{section}
\numberwithin{table}{section}
\numberwithin{figure}{section}
\definecolor{myBlue}{RGB}{113,104,238} % medium slate blue	
\definecolor{myGreen}{RGB}{154,205,50} % olive Drap
\definecolor{myGreen2}{RGB}{114,175,30} 
\definecolor{myRed}{RGB}{180,50,50}  
\definecolor{myOrange}{RGB}{225,92,22} 
\definecolor{lgray}{RGB}{200,200,200} % lightgray
\definecolor{llgray}{RGB}{175,175,175} % lightgray
\newcommand\N{\mathbb N}

\newcommand\R{\mathbb R}

\DeclareMathOperator{\supp}{supp}

\newcommand\calN{\mathcal N}
\newcommand\calO{\mathcal O}
\newcommand\calP{\mathcal P}

\newcommand\calT{\mathcal T}
\newcommand\calV{\mathcal V}

\newcommand\eps{\varepsilon}

\def\dx{\,\text{d}x}

\newcommand\Neigs{N_\text{eigs}}
\newcommand\Kpre{K_\text{pre}}   
\newcommand\Kpost{K_\text{post}}

\newcommand\V{\calV}

\newcommand\Vvert{{|\hskip-0.9pt|\hskip-0.9pt|}}

\definecolor{myRed2}{RGB}{165,42,42} % brown
\definecolor{mycolor5}{rgb}{0.46600,0.67400,0.18800}

\begin{document}
\title[Localized computation of Schr\"odinger eigenstates]{Localized computation of eigenstates of random Schr\"odinger operators}
\author[]{R.~Altmann$^*$, D.~Peterseim$^*$}
\address{${}^{*}$ Department of Mathematics, University of Augsburg, Universit\"atsstr.~14, 86159 Augsburg, Germany}
\email{\{robert.altmann, daniel.peterseim\}@math.uni-augsburg.de}
%\thanks{.}

\date{\today. Accepted for publication in SIAM J. Sci. Comput.}
%\keywords{}
%=============================================================================
%=========  Abstract
%=============================================================================
\begin{abstract}
This paper concerns the numerical approximation of low-energy eigenstates of the linear random Schr\"odinger operator. Under oscillatory high-amplitude potentials with a sufficient degree of disorder it is known that these eigenstates localize in the sense of an exponential decay of their moduli. We propose a reliable numerical scheme which provides localized approximations of such localized states. The method is based on a preconditioned inverse iteration including an optimal multigrid solver which spreads information only locally. The practical performance of the approach  is illustrated in various numerical experiments in two and three space dimensions and also for a non-linear random Schr\"odinger operator.
\end{abstract}
%
%
%=============================================================================
%=========  Title / Contents
%=============================================================================
\maketitle
\indent
{\tiny {\bf AMS subject classifications.}  {\bf 65N25}, {\bf 65N55}, {\bf 47B80}} 
%
%65N25: Numerical analysis -- PDEs, BVPs -- Eigenvalue probems
%65N55: Numerical analysis -- PDEs, BVPs -- Multigrid/Domain decomposition
%47B80: Operator theory -- Special linear operators -- Random operators
%
%35P15: PDEs -- Spectral theory -- Estimation of eigenvalues
%81Q10: Quantum theory -- Mathematical theory -- Selfadjoint op, spectral analysis
%
%
%=============================================================================
%=========  Introduction
%=============================================================================
\section{Introduction}
This paper concerns the numerical approximation of essentially localized eigenstates of the linear Schr\"odinger eigenvalue problem 
\begin{align}
\label{eq:EVP:strong}
  -\Delta u + Vu = \lambda u
\end{align}
on a bounded domain~$D\subseteq \R^d$ with homogeneous Dirichlet boundary conditions. The non-negative variable coefficient $V$ represents an external potential reflecting a high degree of disorder. 
% motivation: BEC, GPE
This apparently simple problem is relevant in the context of quantum-physical processes related to ultracold bosonic or photonic gases, known as Bose-Einstein condensates~\cite{Bos24,Ein24,DalGPS99, PitS03}. A Bose-Einstein condensate (BEC) is an extreme state of matter formed by a dilute gas of bosons at ultra-cold temperatures, very close to absolute zero. In a BEC, individual particles (i.e., their wave packages) overlap, lose their identity, and form one single super atom. BECs allow to study macroscopic quantum phenomena such as superfluity (i.e., the frictionless flow of a fluid) on an observable scale. When BECs are trapped in a highly oscillatory high amplitude potential $V$ that exhibits a sufficiently large degree of disorder, the low-energy stationary states essentially localize in the sense of an exponential decay of their moduli. This localization is a universal wave phenomenon referred to as Anderson localization~\cite{And58}. For BECs, it has been observed experimentally in~\cite{FalFI08}.

On a mathematical level, the formation of stationary quantum states can be modeled by the slightly more involved Gross-Pitaevskii eigenvalue problem (GPEVP). In non-dimensional form, the GPEVP seeks $L^2$-normalised eigenfunctions $u \in H_0^1(D)$ and corresponding eigenvalues (so-called chemical potentials) $\lambda \in \mathbb{R}$  such that 
\begin{align*}
 - \triangle u + V u + \delta |u|^2 u = \lambda u.
\end{align*}
The parameter $\delta\ge0$ resembles the strength of repulsive particle interactions depending on physical properties of the particles that form the BEC. The linear case ~\eqref{eq:EVP:strong} corresponds to the regime of vanishing particle interaction ($\delta=0$). Since the interactions are weak for BECs, the linear case provides very good approximations of actual BEC states~\cite{Rog13}. This is why we will mostly consider the numerical solution of the linear eigenvalue problem \eqref{eq:EVP:strong}. However, the numerical techniques can be generalized to the nonlinear setting which will be demonstrated through numerical experiments at the end of this paper.

The numerical approximation of localized Schr\"odinger eigenstates has recently caused a large interest in the fields of computational physics and scientific computing. Fundamentally novel methodologies have been developed to cope with the intrinsic difficulty coming from the multiscale nature of the potential~\cite{ArnDJMF16,Ste17, ArnDFJM19,XieZO18ppt}. They are based on the link of the groundstate~$u_1$, normalized by~$\| u_1 \|_{L^{\infty}(D)}=1$, and the so-called landscape function~$\psi \in H^1_0(D)$ defined through of the homogeneous elliptic equation~$-\Delta \psi + V\psi=1$, cf.~\cite{FilM12}. The landscape function $\psi$ gives rise to surprisingly sharp eigenvalue bounds and its local maxima indicate regions where localization may occur.  
While the new techniques based on landscape functions are empirically successful in predicting the eigenvalues they lack any control on the accuracy of the approximation. In particular the approximation of the states is rather sketchy. 

The present paper aims at a more sophisticated method leading not only to the regions of localization but also to actual approximations of the lowermost eigenstates. The starting point is the recent rigorous a priori prediction of exponentially localized low-energy states caused by the interplay of disorder (randomness) and high amplitude (contrast) of the potential trap $V$ \cite{AltHP18ppt}. The results of \cite{AltHP18ppt} employ the convergence analysis of a preconditioned inverse iteration in the spirit of iterative numerical homogenization \cite{KorY16} and thereby provides a role model for efficient simulation. 
The main idea is to exploit the localization property, meaning that these eigenfunctions can be approximated in a sophisticated manner by functions supported in the union of only a few small sub-domains. The main contribution of this paper is the finding of an appropriate starting subspace and the construction of a preconditioned eigensolver which only relies on local operations. In this way, we are able to approximate the eigenstates of lowest energy roughly at the same costs as the computation of the landscape function~$\psi$. In addition, the algorithm is applicable for the nonlinear case. 

% algorithm
The first step is a finite element discretization of the eigenvalue problem~\eqref{eq:EVP:strong} which we discuss in Section~\ref{sec:evp}. For this, we consider a uniform refinement of the~$\eps$-mesh on which the potential~$V$ is defined. This automatically provides a mesh hierarchy for which we define a multigrid based preconditioner. Using only a small number of smoothing steps and no direct solves, one multigrid step preserves locality in the sense that the support of the resulting function is only slightly larger than the initial function. This then leads to a {\em localization preserving} eigensolver introduced in Section~\ref{sec:precond}. In other words, the simple yet efficient trick is to consider a multigrid preconditioner on a hierarchy of meshes starting from the $\eps$-level. This sufficed to be optimal in the sense of an $\mathcal{O}(1)$ condition number of the preconditioned operator. If no direct solver is used on the coarsest level, this preconditioner prevents the otherwise global communication of a standard multigrid involving coarser levels. 
The overall algorithm consists of two parts. First, we apply in parallel the localization preserving iteration scheme to finite element basis functions on a coarse grid. Based on the energies we select the most promising candidates indicating the regions of localization. Second, we consider a Ritz-Rayleigh iteration on the remaining local functions, i.e., we project the eigenvalue problem on a very small subspace. This procedure is applicable in any dimension and thus, allows to compute eigenfunctions for~$d=3$ where standard eigensolvers break. This and further numerical experiments are subject of Section~\ref{sec:numerics}.
%
%
%=============================================================================
%=========  Eigenvalue Problem
%=============================================================================
\section{Schr\"odinger Eigenvalue Problem}\label{sec:evp}
This section is devoted to the linear Schr\"odinger eigenvalue problem~\eqref{eq:EVP:strong} and the introduction of needed finite element spaces. Further, we discuss localization effects of the lowermost eigenfunctions if the potential contains disorder. 
%
%
%%%%%%%%%%%%%%%%%%%%%%%%%%%%%%
\subsection{Model problem}\label{sec:evp:model}
In this paper, we consider the~$d$-dimensional linear eigenvalue problem of Schr\"odinger type with a potential  being highly oscillatory and of large amplitude. In particular, we assume that the potential includes some kind of disorder such that the first few eigenfunctions of lowest energy localize. 

The variational formulation corresponding to the eigenvalue problem~\eqref{eq:EVP:strong} reads as follows: Given a non-negative potential $0\leq V\in L^\infty(D)$, find non-trivial eigenpairs $(u, \lambda) \in \V\times\R$ with search and test space~$\V := H^1_0(D)$ such that 
\begin{align}
\label{eq:EVP:weak}
  a(u, v) 
  := \int_{D} \nabla u(x) \cdot \nabla v(x) + V(x)\, u(x) v(x) \dx 
  = \lambda\, (u, v) 
\end{align}
for all test functions $v \in \V$. Here, $(\cdot,\cdot)$ denotes the~$L^2$-inner product on~$D$ and eigenfunctions are assumed to be normalized in the~$L^2$-norm. Further, we assume that the eigenvalue are ordered, i.e., $0 <\lambda_1 < \lambda_2 \le \dots$. 

As far as the theory is concerned, we focus on a representative class of potentials which are piecewise constant with respect to a mesh~$\calT^\eps$ consisting of cubes with side length~$\eps\ll 1$. For the sake of simplicity we assume~$\eps=2^{-\ell_\eps}$ where~$\ell_\eps$ denotes the level of the~$\eps$-scale. On each cube in~$\calT^\eps$, the potential takes (randomly) a value in the interval~$[\alpha, \beta]$ with moderate $0\le \alpha \approx 1$ and large $\beta$ in the sense of~$\beta \gtrsim \eps^{-2}$. We emphasize that the resulting potentials are highly oscillatory due to the underlying mesh on $\eps$-scale. Examples of such random potentials are shown in Figure~\ref{fig:pot}. In the field of matter waves, one often considers disorder potentials created optically by using speckle patterns~\cite{BoiMFGSG99,LyeFMWFI05,FalFI08}. This means that a laser beam is transmitted trough a diffusive plate forming randomized and high-contrast patterns. Within this paper, we imitate such speckles by piecewise constant potentials with respect to a quadrilateral mesh consisting of cubes with side length~$\eps$. On each of these cubes the potential takes random values following certain statistical assumptions, cf.~\cite{Goo75,DunKW08}. 
%
% Bild von Potentialen
\begin{figure}
\includegraphics[width=4.8cm, height=4.8cm]{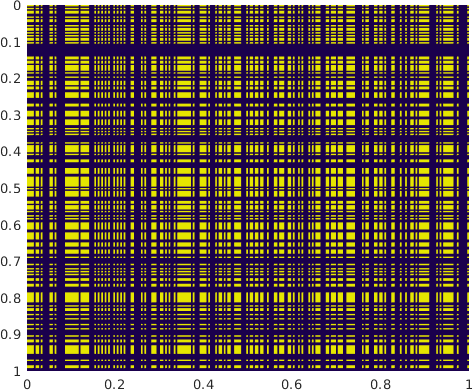}
\includegraphics[width=4.8cm, height=4.8cm]{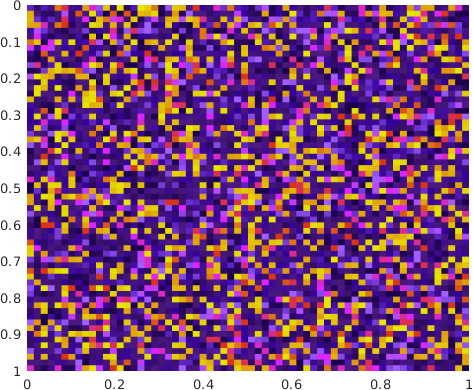}
\includegraphics[width=4.8cm, height=4.8cm]{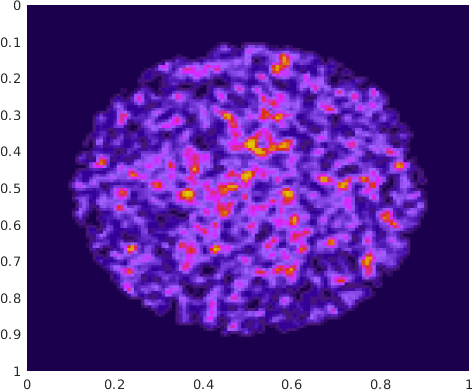}
\caption{Three examples of prototype disorder potentials: a 1D-tensorized potential with only two values (left), a fully random potential (middle), and a speckle potential (right). In all cases, dark color implies a large value of the potential up to the maximum~$\beta$.} 
\label{fig:pot}
\end{figure}

Within this paper, $\Vert\cdot\Vert:=\sqrt{(\cdot,\cdot)}$ denotes the standard~$L^2$-norm on $D$, whereas the~$V$-weighted~$L^2$-norm is given by 
\[
  \Vert v \Vert_V^2 
  := (Vv,v)
  = \int_D V(x)\,  |v(x)|^2\dx. 
\]
Corresponding to the Schr\"odinger operator, we define the energy norm as 
\[
  \Vvert v \Vvert^2
  := a(v, v)
  = \Vert \nabla v\Vert^2 + \Vert v\Vert_V^2.
\]
The energy of a function is characterized by the Rayleigh quotient, namely 
\begin{align*}
%\label{def:Rayleigh}
	\lambda(v) 
	:= \frac{a(v,v)}{\Vert v\Vert^2}
	= \frac{\Vvert v\Vvert^2}{\Vert v\Vert^2}.
\end{align*}
In case $v$ is an eigenfunction of the Schr\"odinger eigenvalue problem~\eqref{eq:EVP:weak}, the energy~$\lambda(v)$ is equal to the corresponding eigenvalue. The eigenfunction of minimal energy~$\lambda_1 := \min_{v\in \V \setminus \{0\}} \lambda(v)$ is called the groundstate. 
%
%
%%%%%%%%%%%%%%%%%%%%%%%%%%%%%%
\subsection{Exponential decay of the Green's functions}\label{sec:evp:green}
For certain classes of potentials it is known that the corresponding Green's function of the Schr\"odinger operator decays exponentially. For constant potentials~$V\equiv \beta$ with sufficiently large~$\beta$ this was shown in~\cite[Lem.~3.2]{Glo11}. 

Piecewise constant potentials with two values~$\alpha$ and~$\beta$ were considered in~\cite{AltHP18ppt}. For this, an operator preconditioner was constructed depending on the geometric structure of the potential, i.e., on the interaction of $\alpha$-valleys and~$\beta$-peaks. It was designed in such a way that the application of this operator preconditioner only enlarges the support by a small amount of~$\eps$-layers within~$\calT^\eps$. With this, one can show that the weak solution~$u\in \calV$ of the variational problem $a(u,\cdot) = (f,\cdot)$ decays exponentially fast around the support of~$f$ for potentials under certain statistical assumptions (including, e.g., 1D-tensorized potentials as in Figure~\ref{fig:pot}). More precisely, this means that 
\[  
	\Vvert u\Vvert_{D\setminus B^\infty_{p\eps L}(\supp f)}
	\le c\, \gamma^{p}\, \Vvert u\Vvert
\]
for constants~$c>0$ and~$0<\gamma<1$ independent of~$\eps$, $L\approx \log(1/\eps)$, and $B^\infty_r(z)$ denoting the ball of radius~$r$ around~$z$. We emphasize that this result only relies on~$\eps$ being small (oscillatory) and~$\beta$ being large (high amplitude). Thus, disorder does not play a role for the exponential decay of the Green's function. 
%
%
%%%%%%%%%%%%%%%%%%%%%%%%%%%%%%
\subsection{Localization of eigenfunctions}\label{sec:evp:localization}
For the localization of eigenfunctions, the potential needs to include a certain degree of disorder. For a periodic potential the lowermost part of the spectrum of the Schr\"odinger operator is clustered. In the one-dimensional case, which is illustrated in Figure~\ref{fig:decay}, one can prove that there exist about~$\eps^{-d}/2$ (the number of potential valleys) eigenvalues in the energy range of~$\eps^{-2}$. Disorder changes the picture dramatically and leads to significant spectral gaps already within the first few eigenvalues. In one space dimension there is a one-to-one correspondence between the largest~$\alpha$-valleys of the potential and the smallest eigenvalues. This is illustrated by dotted vertical lines which correspond to the largest~$\alpha$-valleys and coincide with the spectral gaps for the random potential. This then leads to the exponential decay of the lowermost eigenfunctions, cf.~Figure~\ref{fig:decay}.
%
% Illustration of locla/nonlocal eigenstates
\begin{figure}
\includegraphics[width=4.8cm, height=4.8cm]{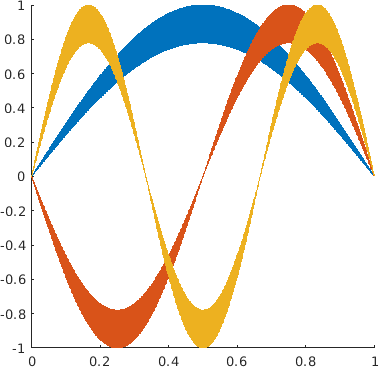}
\includegraphics[width=4.8cm, height=4.8cm]{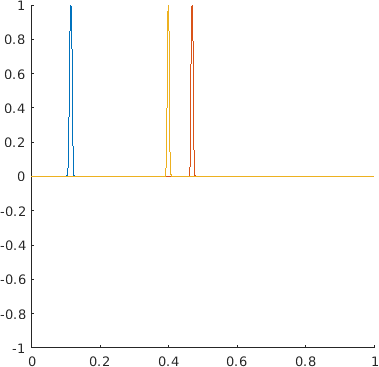}
\includegraphics[width=4.8cm, height=4.8cm]{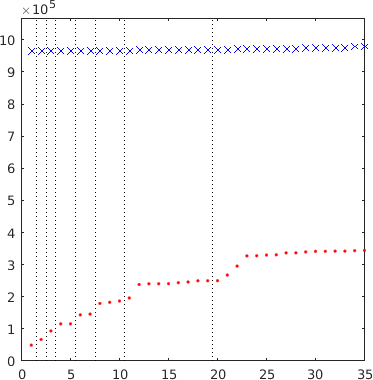}
\caption{The first three eigenfunctions for a periodic (left) and random (middle) potential in the one-dimensional case with~$\ell_\eps = 10$ and $\beta=2\cdot \eps^{-2}$. The lower part of the spectrum for both cases (periodic {\color{blue}$\times$} and random ${\color{red}\cdot}$\,) is shown on the right, clearly depicting the different nature of the two cases. The dotted vertical lines indicate for the random case that there is a single valley of largest, second largest, and third largest diameter, two valleys of fourth largest diameter, and so forth.} 
\label{fig:decay}
\end{figure}

% our proof
In~\cite{AltHP18ppt} the localization of eigenfunctions was rigorously proven for the regime~$\beta\gtrsim\eps^{-2}$ and certain statistical assumptions on the potential. The key ingredient was to prove the existence of spectral gaps in the presence of disorder in combination with a preconditioned block inverse iteration. For this, particular finite element spaces and quasi-interpolation operators were considered. Further, the connection of valley sizes and eigenvalues was used. More precisely, we considered the first eigenfunctions of the Laplacian in the largest~$\alpha$-valleys with homogeneous Dirichlet boundary conditions to obtain eigenvalue estimates of the Schr\"odinger operator. 
These Laplace eigenfunctions also serve as stating functions within the preconditioned block inverse iteration. The claimed decay then follows from the exponential convergence with respect to the number of iteration steps.

% decay by landscape function
The exponential decay may also be characterized a posteriori in terms of the landscape function~$\psi \in \calV$ defined through 
\begin{align}
\label{eqn:defLandscape}
  a(\psi,v)  = (1, v)
\end{align}
for all test functions~$v\in\calV$, cf.~\cite{FilM12}. The connection to the eigenvalue problem~\eqref{eq:EVP:strong} gets visible by the possible reformulation as an eigenvalue problem with the effective confining potential~$1/\psi$, which encodes the decay of the eigenfunction in some Agmon measure~\cite{ArnDJMF16}. This means that the eigenfunction reduces by a certain factor when crossing a valley of~$1/\psi$. This interplay has also been analyzed in~\cite{Ste17} using averages over local Brownian motion paths. 
%
%
%%%%%%%%%%%%%%%%%%%%%%%%%%%%%%
\subsection{Finite element discretization}\label{sec:evp:fem}
In order to approximate the (localized) eigenstates of the Schr\"odinger operator, we consider a finite element discretization of~\eqref{eq:EVP:weak}. For this, we consider a uniform refinement of~$\calT^\eps$, namely~$\calT^h$, consisting of cubes with side length~$h=2^{-\ell_h}\le \eps$. 
The corresponding set of nodes is denoted by~$\calN^h$ and the set of interior nodes by~$\calN^h_0$.  

% V_h
Based on~$\calT^h$, we define~$V_h := Q_1(\calT^h) \cap \calV \subseteq \calV$ as the conforming~$Q_1$-finite element space, cf.~\cite[Sect.~3.5]{BreS08}. This space has the dimension~$n:=|\calN^h_0|$ and is spanned by the standard~$Q_1$ hat functions, which we denote by~$\varphi^h_z$ for~$z\in\calN^h_0$. Recall that~$\varphi^h_z$ is a piecewise polynomial of partial degree one with $\varphi^h_z(z)=1$ and $\varphi^h_z(w)=0$ for any other node $w\in\calN^h\setminus\{z\}$. Clearly, hat functions can also be defined on the original mesh~$\calT^\eps$ w.r.t.~its set of interior nodes~$\calN^\eps_0$. For these hat functions we write~$\varphi^\eps_z$ for~$z\in\calN^\eps_0$. 

% FEM
We now apply a Galerkin ansatz to the eigenvalue problem~\eqref{eq:EVP:weak}. Thus, we restrict the trial and test space to~$V_h$, which then leads to a discrete eigenvalue problem of the form  
\begin{align}
\label{eq:EVP:FEM}
  A_h u_h 
  = \lambda_h M_h u_h.
\end{align}
Here, $A_h \in \R^{n, n}$ and $M_h \in \R^{n, n}$ denote the symmetric stiffness and mass matrices defined through
\[
  (A_h)_{ij} 
  := a(\varphi^h_{z_i}, \varphi^h_{z_j}),\qquad
  (M_h)_{ij} 
  := (\varphi^h_{z_i}, \varphi^h_{z_j}).
\]
Note that the stiffness matrix already includes the potential. Eigenpairs of the matrix eigenvalue problem~\eqref{eq:EVP:FEM} approximate eigenpairs of the PDE eigenvalue problem~\eqref{eq:EVP:weak}. Due to the min-max principle of eigenvalues based on the Rayleigh quotient, we know that~$\lambda_1 \le \lambda_{h,1}$.  

% hard to solve
As mentioned above, we consider highly oscillatory potentials meaning that~$\eps$ is small. Further, the mesh size~$h$ needs to sufficiently small compared to~$\eps$ in order to resolve the oscillations of the potential and thus, guarantees a reasonable approximation of the eigenpairs. Such a condition on the minimal resolution are justified by standard a priori error analysis. If $D$ is convex, any eigenfunction $u\in H^1_0(D)$ is $H^2$ regular. When normalized in $L^2(D)$, $u$ satisfies the bound
\[
  \|D^2u\|
  \leq \|\Delta u\|
  = \|\lambda u - Vu\|
  \leq \lambda+\beta,
\]
where $\lambda$ is the eigenvalue it corresponds to. For the lowermost eigenvalue in the regime of~\cite{AltHP18ppt} where~$\lambda_1\approx\beta\approx\eps^{-2}$ this leads to an error bound
\begin{equation*}
  \frac{\lambda_{h,1}-\lambda_1}{\lambda_1}
%  \lesssim \min_{v_h\in V_h} \Vvert u-v_h \Vvert^2
  \lesssim \frac{h^2}{\eps^2} +  \frac{h^4}{\eps^4}, 
\end{equation*}
see, e.g., \cite[Lem.~6.1]{StrF73}. While the hidden constant in this bound depends only on the quasi-uniformity of the mesh, multiplicative constants in bounds for larger eigenvalues or any eigenfunction may deteriorate with the distance to neighboring eigenvalues in the case of clustered eigenvalues~\cite{MR3107358}. In this sense the resolution condition $h\lesssim\eps$ is minimal for the approximation of the lowermost eigenvalue and may be much more restrictive in other cases. As a result, the eigenvalue problem~\eqref{eq:EVP:FEM} is of large dimension, which makes the solution computationally costly. 

As we are interested in the smallest eigenvalues and the corresponding localized eigenstates, we aim to construct a preconditioner based on local operations. Such a {\em localization preserving} preconditioner then allows parallel computations and is subject of the following section.
%
%
%=============================================================================
%=========  Preconditioner
%=============================================================================
\section{Localization Preserving Preconditioner}\label{sec:precond}
In~\cite{AltHP18ppt} we have presented a locally operating preconditioner based on a domain decomposition of~$D$, which was a key ingredient for the proof of the exponential decay of the first eigenstates. The construction, however, was tailored to theoretical needs and checkerboard potentials. It included exact projections of local sub-domains. For actual computations, it turns out that already a simple multigrid preconditioned iteration may act as a localized version of a block inverse power method leading to sophisticated approximation results. In order to preserve locality, the coarsest level used within the multigrid cycle is the mesh~$\calT^\eps$ on which the potential is defined. This still yields an optimal preconditioner due to the properties of the Green's function of the Schr\"odinger operator.  

Further, we discuss how to find a low-dimensional starting subspace of local functions with which we can initialize the eigenvalue iteration. This is a crucial step in order to obtain local approximations of the lowermost eigenfunctions. 
%
%%%%%%%%%%%%%%%%%%%%%%%%%%%%%%
\subsection{Multigrid preconditioner}\label{sec:precond:multigrid}
Recall that we have assumed that the finite element mesh~$\calT^h$ is defined by a uniform refinement of the mesh on $\eps$-level on which the potential is defined. Thus, we have a mesh hierarchy given by~$\calT^h$, possible intermediate meshes, and~$\calT^\eps$. For this hierarchy, we define a standard geometric multigrid method without direct solve on the coarsest level, which will serve as a preconditioner for the eigenvalue iteration. A major aspect is that the multigrid method maintains locality, i.e., for a local right-hand side~$b$ the resulting approximate solution of~$A_h^{-1}b$ is supported on a domain which is only slightly larger than the support of~$b$. 

% V-cycle
Given the mesh hierarchy, we consider a so-called V-cycle with only one smoothing step on each level, cf.~\cite{Hac85, Mcc87,Yse93}. Starting with a vanishing starting vector, we consider the residuals of~$A_h x = b$ on different levels of the hierarchy. On the coarsest level, i.e., on~$\calT^\eps$, we run one single relaxation step. This means that we do not use a direct solver on~$\eps$-level with the stiffness matrix~$A_\eps$ but use instead a single step of the Jacobi iteration. This yields a reasonable approximation, since the condition number of~$A_\eps$ is of order~$1$. The reason for this is the shift by the potential in the bilinear form~$a$ in~\eqref{eq:EVP:weak} and the fact that~$\beta \gtrsim \eps^{-2}$. A computational study of the condition numbers of the stiffness matrices~$A_h$ and~$A_\eps$ is shown in Figure~\ref{fig:cond}. We emphasize that a direct solver would ruin the locality immediately.

% locality preserving
On each level of the multigrid cycle the support of the iterate grows slightly due to the application of the stiffness matrix on the respective level. In total, the application of this preconditioner enlarges the support by strictly less than three layers of $\eps$-cubes. 

% pcg
It is well-known that multigrid solvers with a reduced tolerance (we only perform one V-cycle with a single relaxation step per level) can be used efficiently as a preconditioner within an external iterative solver, leading to methods such as {\em pcg} or {\em lopcg}, cf.~\cite{KnyN03}. To keep things simple, we concentrate on pcg. We denote the application of~$j$ steps of pcg with the multigrid preconditioner by~$\calP_j$. Thus, $\calP_j(b)$ yields an approximation of~$A_h^{-1} b$, where the accuracy can be controlled by the number of iteration steps. We emphasize that~$\calP_j$ preserves locality in the following sense. 
\begin{theorem}[Localization preserving property of~$\calP_j$]
\label{thm_pcg}
Assume that~$x\in \R^n$ is the coefficient vector of a local function~$u_h \in V_h \subset \calV$. Then, $\calP_j(x)$ is the representative of a function in~$V_h$ with support being at most~$3j$~layers of $\eps$-cubes larger than~$\supp(u_h)$.  
\end{theorem}
\begin{proof}
As already mentioned, the application of a V-cycle as described above affects less than three layers of~$\eps$-cubes. More precisely, the support enlarges at most within a ball of radius~${3\eps-h}$ around~$\supp(u_h)$. The cg step involves another application of the stiffness matrix on the finest level and thus, adds only one layer of~$h$-cubes to the support. In total this leads to a growth of~$3$ layers with side length~$\eps$ per pcg step.
\end{proof}
%
% Bild von Konditionszahlen
\begin{figure}
% This file was created by matlab2tikz.
%
%The latest updates can be retrieved from
%  http://www.mathworks.com/matlabcentral/fileexchange/22022-matlab2tikz-matlab2tikz
%where you can also make suggestions and rate matlab2tikz.
%
\definecolor{mycolor1}{rgb}{0.00000,0.44700,0.74100}%
\definecolor{mycolor2}{rgb}{0.85000,0.32500,0.09800}%
\definecolor{mycolor3}{rgb}{0.92900,0.69400,0.12500}%
\definecolor{mycolor4}{rgb}{0.49400,0.18400,0.55600}%
\definecolor{mycolor5}{rgb}{0.46600,0.67400,0.18800}%
\begin{tikzpicture}

\begin{axis}[%
width=4.3in,
height=1.8in,
at={(0.758in,0.481in)},
scale only axis,
xmin=3,
xmax=10.2,
ymode=log,
ymin=3,
ymax=35000,
yminorticks=true,
axis background/.style={fill=white},
title style={font=\bfseries},
xlabel={level of discretization $\ell_h$},
ylabel={condition number of $A_h$},
legend style={legend cell align=left, align=left, draw=white!15!black, at={(0.21,0.96)}}
]
\addplot [color=mycolor1, very thick, mark=o, mark options={solid, mycolor1}]
  table[row sep=crcr]{%
4	4.59792828499681\\
5	18.6726271064303\\
6	76.4932467897818\\
7	307.959260073372\\
};
\addlegendentry{$\ell_\eps$ = 4}

\addplot [color=mycolor2, very thick, mark=o, mark options={solid, mycolor2}]
  table[row sep=crcr]{%
5	7.39766840724602\\
6	24.5355391445338\\
7	99.330031224846\\
8	399.233775463649\\
};
\addlegendentry{$\ell_\eps$ = 5}

\addplot [color=mycolor3, very thick, mark=o, mark options={solid, mycolor3}]
  table[row sep=crcr]{%
6	12.8243718962935\\
7	44.5591380652133\\
8	180.389346106339\\
9	724.907587215957\\
};
\addlegendentry{$\ell_\eps$ = 6}

\addplot [color=mycolor4, very thick, mark=o, mark options={solid, mycolor4}]
  table[row sep=crcr]{%
7	12.8026202850746\\
8	44.5671300513109\\
9	180.195654055876\\
10	724.504872517572\\
};
\addlegendentry{$\ell_\eps$ = 7}

\addplot [color=mycolor5, very thick, dashed, mark=triangle, mark options={solid, mycolor5}]
  table[row sep=crcr]{%
3	25.4628868875912\\
4	100.897185940403\\
5	402.653822590254\\
6	1609.68497130673\\
7	6437.81070429023\\
8	25750.3139200032\\
9	103000.326853641\\
10	412000.378604127\\
};
%\addlegendentry{without potential}

\end{axis}
\end{tikzpicture}%
\caption{Condition numbers of~$A_h$ for different values of~$\eps = 2^{-\ell_\eps}$ and $h = 2^{-\ell_h}$. This includes~$A_\eps$ if~$\ell_h=\ell_\eps$. The stiffness matrix is based on a random potential in 2D, only taking the values~$\alpha= 0$ or $\beta = 5\cdot \eps^{-2}$ with equal probability. The {\bf{\color{mycolor5} green dashed}} line shows the condition number of~$A_h$ without a potential.} 
\label{fig:cond}
\end{figure}
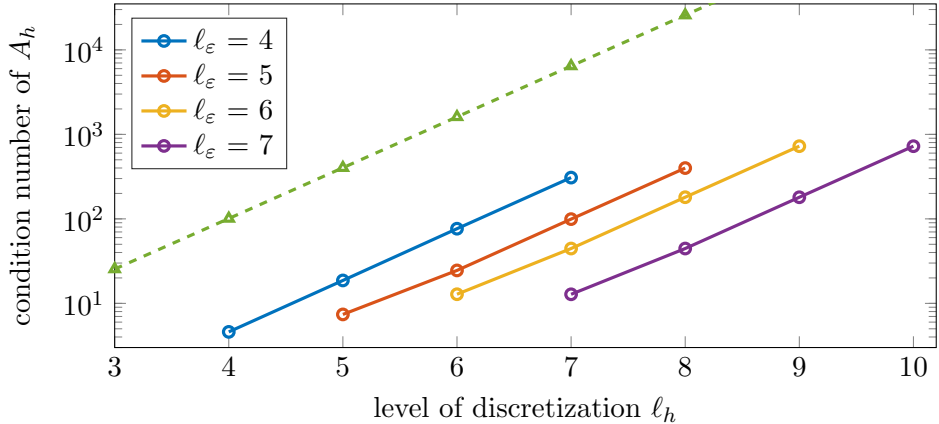
\begin{remark}
Note that the choice of the preconditioner is by far not unique and may be replaced, e.g., by local variants of hierarchical basis~\cite{Yse86} or BPX~\cite{BraPX90} preconditioners or Gamblets~\cite{XieZO18ppt}. Here, local means that~$\calT^\eps$ is the coarsest level in the underlying hierarchy of meshes.
\end{remark}
%
%%%%%%%%%%%%%%%%%%%%%%%%%%%%%%
\subsection{Starting subspace}\label{sec:precond:start}
For an efficient eigenvalue iteration we also need a suitable starting subspace. Starting for example with a vector which has only positive entries, we will approach the groudstate but every iteration step is 'global'. Instead, we search for a number of local functions at positions where we expect localization of the first eigenfunctions. For the detection of a suitable starting subspace we discuss several approaches. 
%
%%%%%
\subsubsection{Landscape function}\label{sec:precond:start:Filoche}
One possibility to find spots of localization is to compute the already mentioned {\em landscape function}~$\psi \in \calV$, cf.~\cite{FilM12}. This function is defined as the solution of the Schr\"odinger source problem with right-hand side~$1$ and homogeneous Dirichlet boundary conditions, cf.~equation \eqref{eqn:defLandscape}. In other words, $\psi$ is the outcome of one step of the inverse power method in the PDE setting. The corresponding discrete approximation~$\psi_h$ satisfies 
\[
  \psi_h = A_h^{-1} M_h\, \mathbf{1}, 
\]
where $\mathbf{1} = [1,\ 1,\ \dots, 1]^T \in \R^{n}$. The landscape function shall indicate where the first eigenfunctions may localize. To see this, let~$u_1$ denote the normalized ground state of the Schr\"odinger eigenvalue problem, i.e., ~$\| u_1 \|_{L^{\infty}(D)}=1$. Then, it is shown in~\cite{FilM12} that the landscape function satisfies pointwise 
\begin{align}
\label{eq:Filoche}
  |u_1(x)| \le \lambda_1 |\psi(x)|. 
\end{align}
This bound implies that in regions where~$\psi$ is small compared to the smallest eigenvalue~$\lambda_1$, the eigenfunction $u_1$ needs to be small as well. Considering the peaks of the landscape function then provides empirically accurate predictions where to expect localization. However, since we know that~$\lambda_1 \gtrsim \eps^{-2}$, the estimate may degenerate quickly to~$|u_1(x)|\le 1 =\| u_1 \|_{L^{\infty}(D)}$. Thus, this approach does not allow rigorous predictions a priori but may serve as an indicator. 

Although the estimate~\eqref{eq:Filoche} only includes the groundstate, the landscape function has been used to locate a larger number of eigenfunctions by analyzing the local minima of~$\psi$, cf.~\cite{ArnDFJM19}. Further, it may be used to partition the computational domain~$D$ into a network of valleys, cf.~Figure~\ref{fig:landscape}. Yet another possibility is to consider multiple applications of the inverse Schr\"odinger operator, i.e., $(A_h^{-1}M_h)^k\mathbf{1}$ in the discrete setting, cf.~\cite{Ste17}.
%
% FILOCHE - different color schemes: A, B, C
% parameters: eps=7, h=9, beta=5*eps^{-2}, rng(45), rand-potential
\begin{figure}
\includegraphics[width=4.8cm, height=4.8cm]{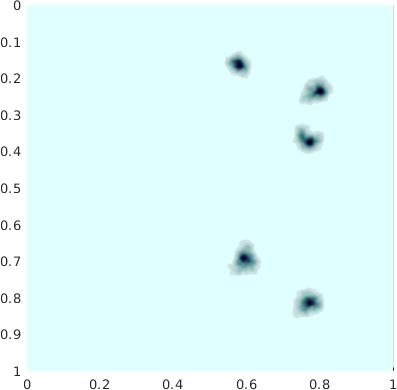}
\includegraphics[width=4.8cm, height=4.8cm]{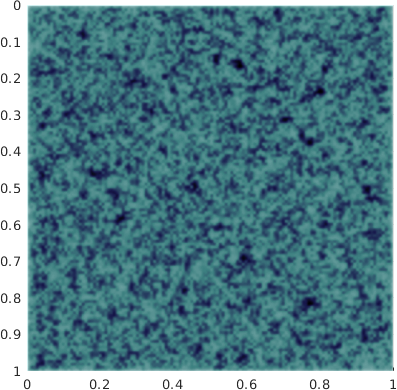}
\includegraphics[width=4.8cm, height=4.8cm]{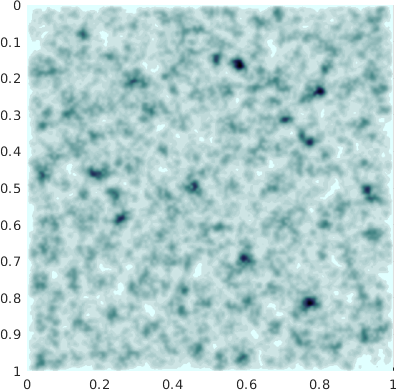}
\caption{Sum of the first five eigenfunctions (left) for a two-dimensional random potential with~$\ell_\eps=7$, $\beta = 5\cdot\eps^{-2}$ and the corresponding landscape function~$\psi$ (middle). Four additional steps of the inverse power method leads to a more selective landscape (right).} 
\label{fig:landscape}
\end{figure}
%
%%%%%
\subsubsection{Randomized landscape function}\label{sec:precond:start:Steinerberger}
A similar approach to detect positions of localization was introduced in~\cite{LuS18}. Therein, the inverse power method is applied to the unit vectors of dimension~$n$. More precisely, for a fixed parameter~$p\in\N$ and~$e_k\in\R^n$ denoting the~$k$-th unit vector, one computes for~$k=1,\dots,n$
\[
  f_p(k) = \log\big( \Vert (A_h^{-1}M_h)^p e_k \Vert_2 \big).
\]
In~\cite{LuS18} it was shown that highly localized eigenfunctions correspond to metastable states of the power iteration and thus, are directly connected to local maxima of the function~$f_p$. Note that in the given setting~$f_p$ computes the inverse iteration to all hat functions on the $h$-scale, which is expensive due to the size of the eigenvalue problem. 

A variant is the following randomized version: Given a random matrix~$R\in\R^{n,m}$ with $m\ll n$ we compute  
\[
  f_{R,p}(k) = \log\big( \Vert e_k^T (A_h^{-1}M_h)^p R \Vert_2 \big).
\]
Thus, we apply the inverse power method to random vectors and consider the means in each component. In other words, we replace the application of~$A_h^{-1}M_h$ to~$n$ local functions by the application to~$m$ global functions~\cite{LuS18}. 
The outcome of this approach is displayed in Figure~\ref{fig:steinerberger}. It shows a landscape function which seems more smoothed than the previous approach of Section~\ref{sec:precond:start:Filoche}. Further, an increase of the number of iterations does not focus purely on the groundstate as this would happen by applying the inverse power method to the vector~$M_h\, \mathbf{1}$. 
%
% STEINERBERGER - different color schemes: A, B, C
% parameters: eps=7, h=9, beta=5*eps^{-2}, rng(45), rand-potential
\begin{figure}
\includegraphics[width=4.8cm, height=4.8cm]{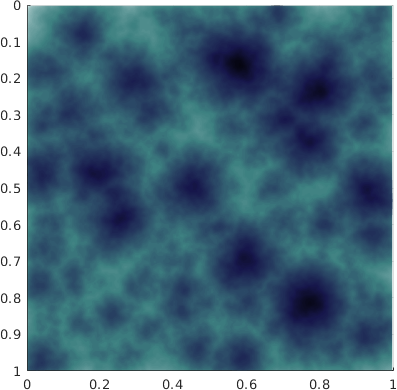}
\includegraphics[width=4.8cm, height=4.8cm]{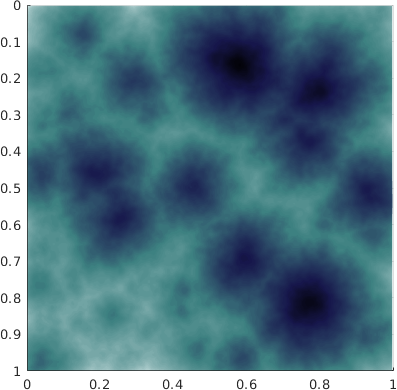}
\includegraphics[width=4.8cm, height=4.8cm]{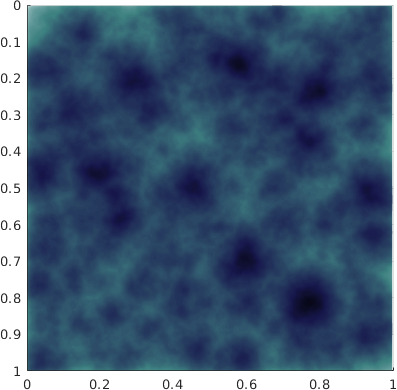}
\caption{Examples of randomized landscape functions~$f_{R,p}$ using~$m$ random vectors and~$p$ applications of the inverse power method (left:~$m=1$, $p=50$; middle:~$m=3$, $p=100$) with the same potential as used in Figure~\ref{fig:landscape}. The plot on the right shows the result for~$m=3$ and~$p=1000$ if we replace~$A_h^{-1}$ by~$I_n - A_h/\Vert A_h\Vert$.}
\label{fig:steinerberger}
\end{figure}
%
%%%%%
\subsubsection{Largest valleys}
Assume that the potential~$V$ is given in form of a random checkerboard, i.e., the function values of~$V$ are either~$\alpha$ or~$\beta$, each with probability~$0.5$. Then, motivated by the theoretical findings in~\cite{AltHP18ppt}, we may start with a number of hat functions distributed in the largest valleys. Here, a valley denotes a cube in which the potential is constant~$\alpha$. 
For special classes of such checkerboard potentials, this approach guarantees to yield an appropriate starting subspace, meaning that the number of initial functions is~$\calO(1)$, that these functions are local, and that the inverse power method converges quickly, cf.~\cite{AltHP18ppt}.  

However, the search for valleys comes with the drawback that one first needs to analyze the structure of the given potential. Further, a generalization of the term valley is needed to consider general potentials. Nevertheless, the idea motivates the subsequent approach using uniformly distributed hat functions. 
%
%%%%%
\subsubsection{Uniformly distributed hat functions}\label{sec:precond:start:hats}
The current approach combines the benefits of the previous subsections. In this manner, we obtain a suitable subspace of local functions which is highly eligible to serve as a starting point within a (preconditioned) eigenvalue iteration. At the same time, we restrict the computational costs such that it is comparable to a single step of the inverse power method. 

% coarse mesh
Let~$\calT^H$ be a mesh of cubes with side length~$H=2^{-\ell_H}\ge\eps$ such that~$\calT^\eps$ is a refinement of~$\calT^H$ and thus, $\ell_H\le \ell_\eps$. For each interior node~$z\in\calN^H_0$ there exists a~$Q_1$ hat function~$\varphi_z^H$. Similar to the approach in Section~\ref{sec:precond:start:Steinerberger} we may now apply several steps of the inverse power method. Instead, we apply a single pcg iteration step as introduced in Section~\ref{sec:precond:multigrid} to each hat function. Note that all these computations may be performed in parallel and that the application of~$\calP_1$ maintains locality in the sense that the support of~$\calP_1(\varphi_z^H)$ is only slightly larger than the support of~$\varphi_z^H$, cf.~Theorem~\ref{thm_pcg}. Although we perform only one pcg step, we gain sufficient information in order to rank the importance of the basis function. For this we compute the energies~$\lambda(\calP_1(\varphi_z^H))$ and sort them in an increasing order. Afterwards, we only keep the functions of lowest energy defined through a fixed ratio~$0<\eta<1$. This procedure is then repeated until the number of functions is small enough. 

% costs, result
We emphasize that the support of each function only grows gently in each step whereas the number of functions decreases by the prescribed factor~$\eta$. In other words, the proposed algorithm is approximately as costly as computing the landscape function~$\psi$, cf.~Section~\ref{sec:precond:start:Filoche}. Here, however, we directly obtain a number of candidates where the first eigenfunctions will localize without the need of constructing a network structure. An illustration of the algorithm is given in Figure~\ref{fig:hats_pre}. 
%
% PRE-iteration - different color schemes: A, B, C
% parameters: eps=7, h=9, H=6, beta=5*eps^{-2}, rng(45), rand-potential
\begin{figure}
\includegraphics[width=4.8cm, height=4.8cm]{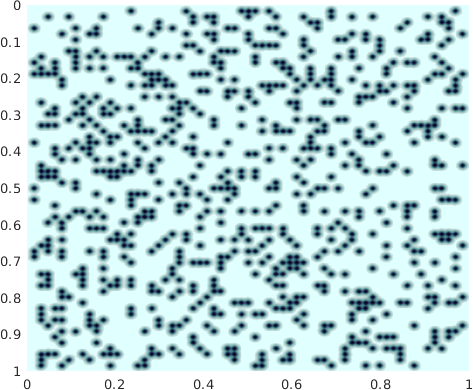}
\includegraphics[width=4.8cm, height=4.8cm]{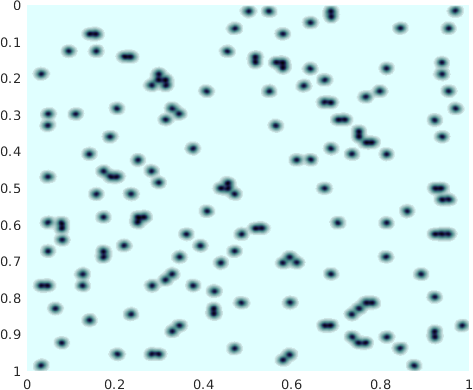}
\includegraphics[width=4.8cm, height=4.8cm]{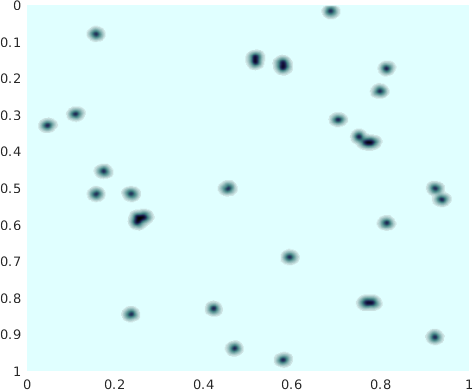}
\caption{Illustration of finding an appropriate local and low-dimensional starting subspace. Starting point are the hat functions on a coarser grid~$\calT^H$. After each pcg step a fixed amount of functions is selected. Pictures show the results for~$\eta=0.2$ after one (left), two (middle), and three (right) pcg steps.}
\label{fig:hats_pre}
\end{figure}
%
%%%%%%%%%%%%%%%%%%%%%%%%%%%%%%
\subsection{Approximation of eigenfunctions}\label{sec:precond:alg}
In this subsection we finally introduce an algorithm to approximate the eigenfunctions of lowest energy of the linear Schr\"odinger eigenvalue problem~\eqref{eq:EVP:weak} under disorder potentials. Thus, we do not only intend to find possible regions of localization but actually compute approximations of the smallest eigenvalues and their corresponding eigenstates. 

Recall the definition of~$\calP_j$ in Section~\ref{sec:precond:multigrid} containing~$j$ pcg steps preconditioned my a multigrid cycle. Further, we fix the following parameters. \medskip

\begin{center}
\begin{tabular}{c||c||l}
parameter & default value & meaning \\[0.1\normalbaselineskip] \hline\hline
$\Neigs$ & $5$ & number of eigenfunctions to compute\\[0.15\normalbaselineskip]
$\Kpre$ & $3$ & number of pre-iteration steps\\[0.15\normalbaselineskip]
$\Kpost$ & $5$ & number of post-iteration steps\\[0.15\normalbaselineskip]
$\ell_H$ & $\ell_\eps- \lfloor \log(\ell_\eps) \rfloor$ & level of~$\calT^H$ for initial hat functions\\[0.15\normalbaselineskip]
$\eta$ & $0.2$ & ratio for selection process\\
\end{tabular}
\end{center}
\vspace{0.3cm}
%
% pre-iteration
The proposed algorithm consists of two parts: In the {\em pre-iteration}, we apply the selection process described in Section~\ref{sec:precond:start:hats}. For this, we consider hat functions on a (coarser) mesh of level~$\ell_H$. Using the hat functions corresponding to~$\calT^H$ rather than~$\calT^\eps$ reduces the number of initial functions. Further, we truncate these hat functions in order improve the locality. After~$\Kpre$ iteration steps with ratio~$\eta$ we then have a much smaller number of functions, which already give a rough approximation of certain eigenfunctions. Most notably, however, they indicate with high precision where localization will take place without the need of constructing an additional mesh or finding local minima as with the landscape approach. If we are only interested in the groundstate, then one may reduce the number of functions until only a single function is left. Otherwise, we need to keep a sufficiently large amount of functions within the iteration process in order to prevent that only the groundstate is approximated. Further, since we only perform rough approximations of the inverse power method, we cannot guarantee that the functions of lowest energy correspond to the lowermost eigenstates. Therefore, we always keep at least~$3\,\Neigs$ functions within the pre-iteration.

% post-iteration
In the {\em post-iteration}, we combine all functions which were selected by the pre-iteration in form of a Ritz-Rayleigh approximation. This means that we first perform three preconditioned cg steps to each function and then project the mass and stiffness matrices to the span of these functions and solve a (small) eigenvalue problem of the form~$\mathbf{M} \alpha = \mu \mathbf{A} \alpha$. Note that the dimension of the eigenvalue problem depends on the number of pre-iteration steps. The first eigenvector~$\alpha_1$ contains the coefficients of the optimal linear combination of the ansatz functions and provides the approximate groundstate. The second eigenvector~$\alpha_2$ then defines the function which serves as approximation of the second eigenstate and so on. Further, we decrease the number of ansatz functions by cutting off the candidates of highest energy. Similar to the pre-iteration, we always keep a certain number of functions, namely twice the amount of eigenfunctions we are interested in, i.e., $2\,\Neigs$. Note that, in contrast to the parallel steps within the pre-iteration, we consider here a combined iteration of all remaining candidates. 

A summary of the procedure is given in Algorithm~\ref{alg:hats} and obtained numerical results for the two- and three-dimensional case are discussed in the subsequent section. 
\begin{algorithm}
\setstretch{1.15}
\caption{Approximation of first~$\Neigs$ localized eigenstates}
\label{alg:hats}
\begin{algorithmic}[1]
	\State {\bf Input}: $\Neigs$, $\Kpre$, $\Kpost$, $\ell_H$, $\eta$
	\vspace{0.5em}
	\State define mesh $\calT^H$ with $H = 2^{-\ell_H}$
	\State $m = |\calN_0^H|$
	\State coefficient matrix of hat functions $\Phi = [\varphi_1^H, \dots, \varphi_m^H] \in \R^{n,m}$
	\vspace{0.5em}
	\For{$k=1$ {\bf to} $\Kpre$} \Comment start pre-iteration
		% a single pcg step
%		\State $\varphi_j = \calP_1(M_h \varphi_j)$, $j=1,\dots,m$
		\State $\Phi = \calP_1(M_h \Phi)$ \label{line:pinvit}
		% normalization
		\State $\Phi_j = \Phi_j / \Vert \Phi_j\Vert_{M_h}$,\ $j=1,\dots, m$
		% sort vector of energies
		\State $\text{energies} = \text{diag} \big( \Phi^T A_h \Phi \big)$
		\State $[\text{energies}, \text{idx}] = \text{sort} (\text{energies})$
		% update m = number of (hat) functions
		\State $m = \max(\lfloor\eta m\rfloor, 3\,\Neigs )$
		% select functions of lowest energy
		\State $\Phi = [\Phi_{\text{idx(1)}}, \dots, \Phi_{\text{idx(m)}}]$ 
	\EndFor
	\vspace{0.5em}
	\For{$k=1$ {\bf to} $\Kpost$} \Comment start post-iteration
		%\State $s = min(3, \Kpost+k)$
		\State $\Phi = \calP_3(M_h \Phi)$ 
		% normalization
		\State $\Phi_j = \Phi_j / \Vert \Phi_j\Vert_{M_h}$,\ $j=1,\dots, m$
		% Ritz-Rayleigh matrics
		\State $\mathbf{M} = \Phi^T M_h \Phi \in \R^{m,m}$ 
		\State $\mathbf{A} = \Phi^T A_h \Phi \in \R^{m,m}$ 
		\State solve eigenvalue problem $\mathbf{M} \alpha = \mu \mathbf{A} \alpha$ \label{line:evp}
		\State resulting eigenpairs~$(\alpha_1, \mu_1), \dots, (\alpha_m, \mu_m)$ with $\mu_1 \le \dots \le \mu_m$ \label{line:cutoff}
		% update m 
		\State $m = \max(\lfloor\eta m\rfloor, 2\,\Neigs )$
		% select functions of lowest energy
		\State $\Phi = [\Phi \alpha_1, \dots \Phi\alpha_m]$ 
\EndFor 
	
\end{algorithmic}
\end{algorithm}
\begin{remark}
The simple pcg steps in the pre-iteration of Algorithm~\ref{alg:hats} (line~\ref{line:pinvit}) may also be replaced by~$\Kpre$ steps of lopcg, cf.~\cite{Kny00,Kny01}. Since the main purpose of the pre-iteration is the detection of regions of localization, we consider here only this simple variant.  
\end{remark}
\begin{remark}
In order to accelerate the iteration and improve the stability of the method, one may consider additional cut-offs. In particular, one may set components of $\alpha_j$ in line~\ref{line:cutoff} of Algorithm~\ref{alg:hats} to zero if they are beneath a certain threshold. This also increases the level of locality. 	 
% \todo{cut off?}% uAugs = uAugs.*(abs(uAugs)>tolPrec*1e-1);
\end{remark}
The procedure of Algorithm~\ref{alg:hats} assumes that the first eigenfunctions are indeed localized. Thus, global eigenstates as they would appear for periodic potentials are not well-approximated by this method. However, one may adapt the selection process in the algorithm to such an extent that periodic structures are at least detected. For this, one may substitute the fixed-ratio criterion by a selection step which is based on the energies of the considered functions. In a periodic structure, where all hat functions would have a comparable energy, we would then not cut off any functions, meaning that the parameter~$m$ does not decrease and that the eigenvalue problem in line~\ref{line:evp} of the algorithm is still of large dimension. 
%
%
%=============================================================================
%=========  Numerics
%=============================================================================
\section{Numerical Examples}\label{sec:numerics}
We perform several numerical tests in order to explore the performance and feasibility of the proposed method. In the first experiment we consider the two-dimensional random potential, which was already used in Section~\ref{sec:precond} for the illustration of the approaches to find regions of localization. Second, we consider a three-dimensional speckle potential leading to huge eigenvalue problems which quickly overcharge standard eigenvalue solvers. Finally, we apply the method to the nonlinear Gross-Pitaevskii eigenvalue problem. 

All computations have been performed with Matlab on an {\em HPC Infiniband cluster} (1.7 TB RAM, 2 Tesla V100 32GB GPUs). The reference solutions, which are used for the error plots, are obtained by the Matlab solver~{\em eigs} with tolerance~$10^{-12}$ for the 2D examples and~$10^{-10}$ in 3D, both with a maximum of~$100$ iterations. 
The code  is available as supplementary material. 
%
%
%%%%%%%%%%%%%%%%%%%%%%%%%%%%%%
\subsection{Random potential in 2D}
In this first numerical experiment we consider a random potential in two space dimensions with parameters
\[
  \ell_\eps = 7, \quad
  \ell_h = 9, \quad
  \ell_H = 6, \quad
  \alpha = 1, \quad
  \beta = 5\cdot\eps^{-2}.
\]
We are interested in the first five eigenstates of lowest energy. The outcome of the pre-iteration with~$\eta=0.2$ was already illustrated in Figure~\ref{fig:hats_pre}, leading to a small number of regions in which we expect the smallest eigenstates to localize. Applying the Ritz-Rayleigh method and the additional selection process, we end up in a ten-dimensional subspace of~$V_h$, where all basis functions have local support. The five functions with lowest energy then serve as approximations of the eigenfunctions~$u_1, \dots, u_5$. The results are very convincing and depicted in Figure~\ref{fig:hats_post}. For a comparison with the actual eigenstates we refer to Figure~\ref{fig:landscape} (left). 
%
% POST-iteration - different color schemes: A, B, C
% parameters: eps=7, h=9, beta=5*eps^{-2}, rng(45), rand-potential
\begin{figure}
\includegraphics[width=4.8cm, height=4.8cm]{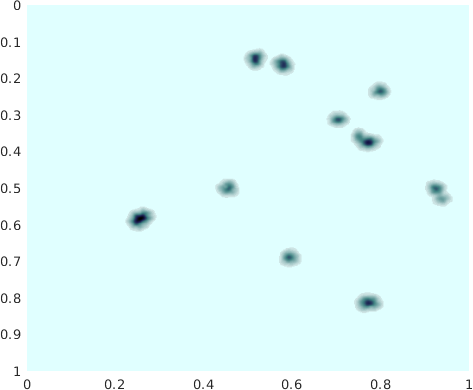}
\includegraphics[width=4.8cm, height=4.8cm]{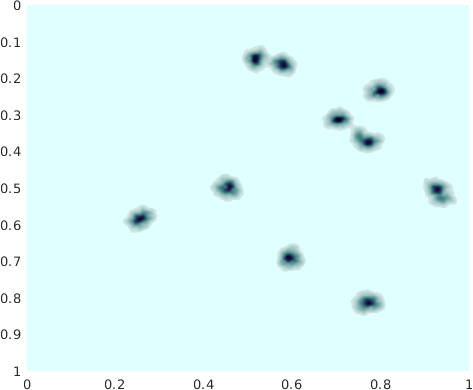}
\includegraphics[width=4.8cm, height=4.8cm]{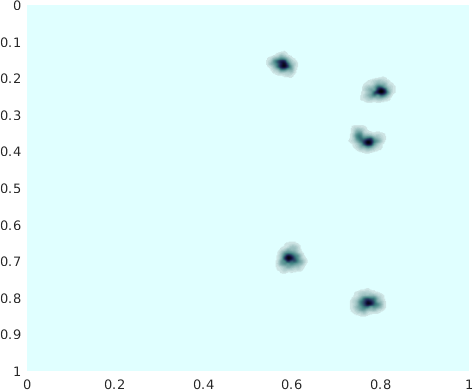}
\caption{Outcome of Algorithm~\ref{alg:hats} with the default parameters after one (left), three (middle), and five (right) Ritz-Rayleigh steps. The two pictures on the left show the sum of all remaining functions whereas the picture on the right only includes the five functions of lowest energy. Cf.~Figure~\ref{fig:landscape} for a reference solution.}
\label{fig:hats_post}
\end{figure}

As mentioned above, the pre-iteration only provides very rough approximations of the eigenstates and serves more the finding of an appropriate starting subspace for the Rayleigh-Ritz method. As a result, the assignments of the localization regions to the actual eigenstates may be inaccurate in the beginning. In the present example, four Rayleigh-Ritz steps were needed until the fifth eigenstate was correctly assigned. This effect can also be observed in the convergence plot in Figure~\ref{fig:conv}. Therein, the relative error in~$\lambda_5$ drops significantly from step~$6$ to step~$7$. Such effects are directly influenced by gaps within the spectrum of the Schr\"odinger operator. Here, the fifth and sixth eigenvalues are given by~$2.6076\cdot 10^4$ and $2.6093 \cdot 10^4$ and thus, at close quarters.
%
% CONVERGENCE plot
% parameters: eps=7, h=9, beta=5*eps^{-2}, rng(45), rand-potential
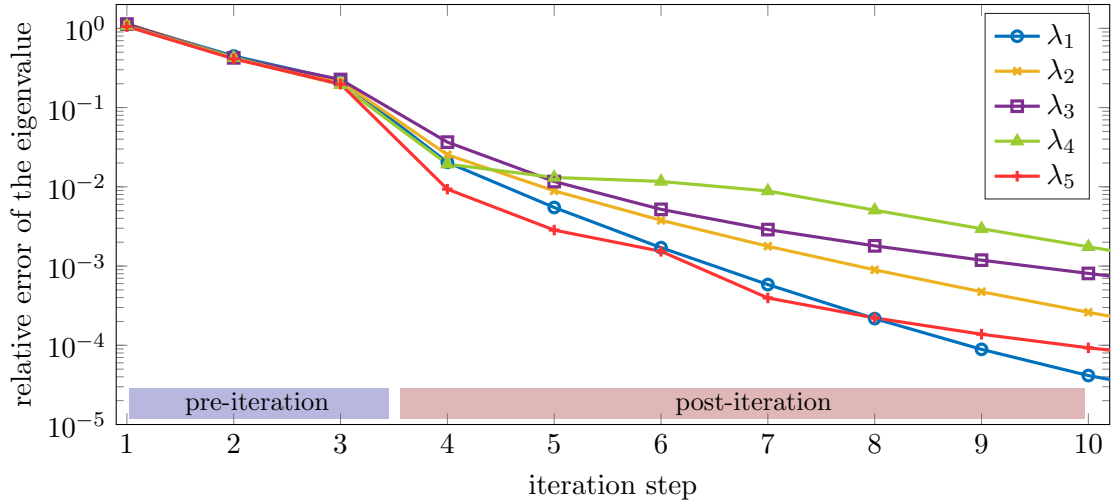
\begin{figure}
% This file was created by matlab2tikz.
%
%The latest updates can be retrieved from
%  http://www.mathworks.com/matlabcentral/fileexchange/22022-matlab2tikz-matlab2tikz
%where you can also make suggestions and rate matlab2tikz.
%
\definecolor{mycolor1}{rgb}{0,0.447,0.741} %blau
\definecolor{mycolor2}{rgb}{0.92900,0.69400,0.12500} %gelb
\definecolor{mycolor3}{rgb}{0.494,0.184,0.556} %lila
\definecolor{mycolor4}{rgb}{0703,0.195,0.195} %red
\definecolor{mycolor5}{rgb}{0.602,0.8,0.195} %green
\begin{tikzpicture}

\begin{axis}[%
width=5.2in,
height=2.2in,
at={(0.758in,0.481in)},
scale only axis,
xmin=0.9,
xmax=10.2,
ymode=log,
ymin=1e-05,
ymax=2,
yminorticks=true,
xlabel={iteration step},
ylabel={relative error of the eigenvalue},
axis background/.style={fill=white},
legend style={legend cell align=left, align=left, draw=white!15!black}
]
\addplot[color=mycolor1, very thick, mark=o, mark options={solid, mycolor1}]
  table[row sep=crcr]{%
1	1.11809701128542\\
2	0.447863593912641\\
3	0.222008486795527\\
4	0.0203691916340474\\
5	0.00548669842870558\\
6	0.00170713373196864\\
7	0.00058409557928613\\
8	0.000217086838900512\\
9	8.88128851745432e-05\\
10	4.15287993897018e-05\\
11	2.33250213978689e-05\\
12	1.6074500030499e-05\\
13	1.31190632778257e-05\\
};
\addlegendentry{$\lambda_1$}

\addplot [color=mycolor2, very thick, mark=x, mark options={solid, mycolor2}]
  table[row sep=crcr]{%
1	1.14509724290677\\
2	0.428456310326161\\
3	0.216480213709656\\
4	0.0253855553670962\\
5	0.00889961390820672\\
6	0.00378094448074982\\
7	0.00177899578727905\\
8	0.000895335021677764\\
9	0.000473792871379199\\
10	0.000260898290704115\\
11	0.000148594986623768\\
12	8.73963379965306e-05\\
13	5.32815244666687e-05\\
};
\addlegendentry{$\lambda_2$}

\addplot[color=mycolor3, very thick, mark=square, mark options={solid, mycolor3}]
  table[row sep=crcr]{%
1	1.13316576360569\\
2	0.424440896918019\\
3	0.226241484599292\\
4	0.036669071171522\\
5	0.0116742428080186\\
6	0.00519721000728162\\
7	0.00288532158681638\\
8	0.00179855565861745\\
9	0.00118704543191072\\
10	0.000806984367172902\\
11	0.000558176852978596\\
12	0.00039074849416528\\
13	0.000276351730870785\\
};
\addlegendentry{$\lambda_3$}

\addplot [color=mycolor5, very thick, mark=triangle, mark options={solid, mycolor5}]
  table[row sep=crcr]{%
1	1.07321837037324\\
2	0.424871732737029\\
3	0.194175919362505\\
4	0.0193404169694789\\
5	0.013135489760899\\
6	0.0116861473505683\\
7	0.00886810555336363\\
8	0.00507175236453924\\
9	0.00295866038527987\\
10	0.00175077334326421\\
11	0.00104834429293744\\
12	0.000634791215232589\\
13	0.000388990362802362\\
};
\addlegendentry{$\lambda_4$}

\addplot[color=mycolor4, very thick, mark=+, mark options={solid, mycolor4}]
  table[row sep=crcr]{%
1	1.05488066780747\\
2	0.412354332135557\\
3	0.197561850483418\\
4	0.00934426710075177\\
5	0.00284358081423333\\
6	0.00153506706333368\\
7	0.00039581618697217\\
8	0.000221614977317625\\
9	0.000137917300584423\\
10	9.28922998045962e-05\\
11	6.67418378181651e-05\\
12	5.06668559652952e-05\\
13	4.03408100547789e-05\\
};
\addlegendentry{$\lambda_5$}
\end{axis}
%
% pre/post iteration
\fill[gray!60!blue, opacity=0.4] (2.1,1.3) -- (5.55,1.3) -- (5.55,1.7) -- (2.1,1.7);
\node at (3.8,1.5) {\small pre-iteration};
\fill[gray!60!red, opacity=0.4] (5.7,1.3) -- (14.8,1.3) -- (14.8,1.7) -- (5.7,1.7);
\node at (10.4,1.5) {\small post-iteration};
\end{tikzpicture}%
\caption{Convergence history of the relative error of the first five eigenvalues~$\lambda_1,\dots,\lambda_5$. Iteration steps 1--3 are the pre-iteration, the remaining steps are part of the post-iteration of Algorithm~\ref{alg:hats}.}
\label{fig:conv}
\end{figure}

Figure~\ref{fig:spectrum} shows that the proposed methods also works if we are interested in a larger number of eigenvalues and eigenfunctions. If we only consider the pre-iteration, the relative error of the eigenvalues seems to be almost constant which is in agreement with the observations in~\cite{ArnDFJM19}. In contrast to the landscape function approach we are able to control the accuracy of the method by the number of post-iterations. Further, the method is flexible enough to recognize when there are several eigenfunctions located in a single valley. 
%
% First 1/eps eigenvalues
% parameters: eps=7, h=9, beta=5*eps^{-2}, rng(45), rand-potential
\begin{figure}
\input{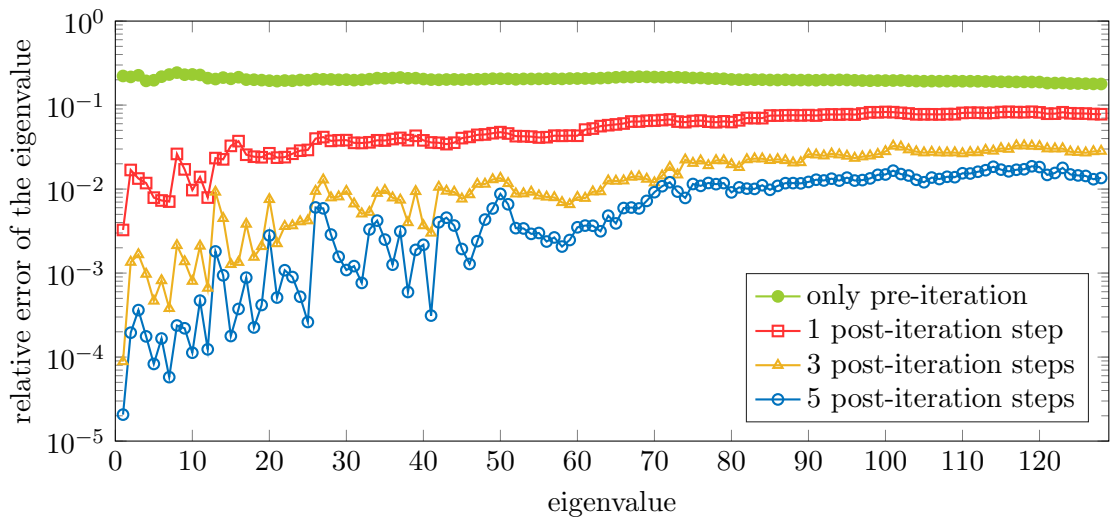}
\caption{Reliability of the method for a large number ($1/\eps=128$) eigenvalues.}
\label{fig:spectrum}
\end{figure}
%
%
%%%%%%%%%%%%%%%%%%%%%%%%%%%%%%
\subsection{Speckle potential in 3D}
The second example considers the linear Schr\"odinger eigenvalue problem~\eqref{eq:EVP:strong} in three space dimensions with a speckle-like potential as it is used in actual experiments~\cite{FalFI08}. As parameters we choose 
\[
  \ell_\eps = 5, \quad
  \ell_h = 6, \quad
  \ell_H = 4, \quad
  \alpha = 1, \quad
  \beta = 5\cdot\eps^{-2}.
\]
An illustration of the potential is given in Figure~\ref{fig:eigComp3D}. As for the previous example we compare the outcome of Algorithm~\ref{alg:hats} with the results obtained by the {\em eigs} solver in Matlab. As a consequence, we are not able to exceed level~$\ell_h=6$, as this marks the limit of the integrated eigenvalue solver. 

The approximation of the first eigenfunction and the corresponding result obtained by {\em eigs} are compared in Figure~\ref{fig:eigComp3D}. The displayed approximations are computed with~$\Kpre= 3$ pre- and~$\Kpost=5$ post-iteration steps. A comparison of computation times for various mesh levels of~$\calT^\eps$ and $\calT^h$ are summarized in the following table: \medskip

\begin{center}
\begin{tabular}{l||r|r|r||r|r|r||}
	& \multicolumn{3}{|c||}{$\ell_h = \ell_\eps+1$} & \multicolumn{3}{c||}{$\ell_h = \ell_\eps+2$} \\ \hline
	& $\ell_\eps = 4$ & $\ell_\eps = 5$\phantom{i} & $\ell_\eps = 6$\phantom{n} & $\ell_\eps = 3$ & $\ell_\eps = 4$\phantom{i} & $\ell_\eps = 5$\phantom{n} \\[0.1\normalbaselineskip] \hline\hline
	Matlab {\em eigs} & $7.15\,s$ & $279.10\,s$ & $\infty\phantom{nn}$ & $7.21\,s$ & $228.60\,s$ & $\infty\phantom{nn}$ \\[0.15\normalbaselineskip]
	Algorithm~\ref{alg:hats} & $4.86\,s$ & $57.65\,s$ & $1788.02\,s$ & $6.43\,s$ & $74.99\,s$ & $1242.41\,s$ \\
\end{tabular}
\end{center}
%
% 3D speckle potential and eigenfunctions
% parameters: eps=5, h=6, beta=5*eps^{-2}, rng(1281), speckle-potential
\begin{figure}
\includegraphics[width=4.8cm, height=4.8cm]{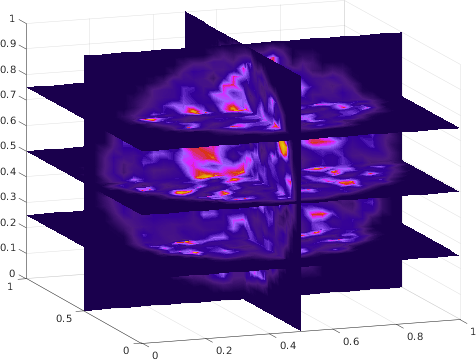}
\includegraphics[width=4.8cm, height=4.8cm]{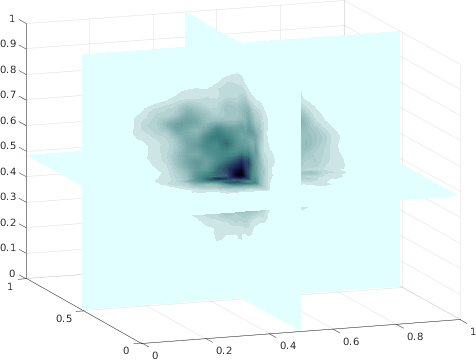}
\includegraphics[width=4.8cm, height=4.8cm]{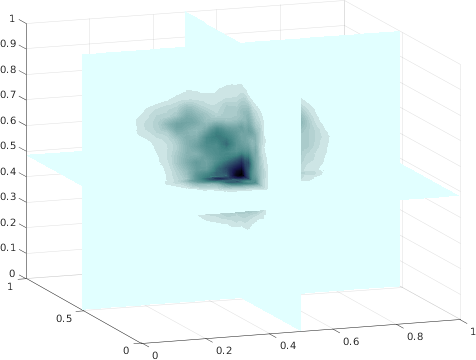}
\caption{Speckle potential (left) and comparison of the first eigenfunction (middle) and its numerical approximation obtained by Algorithm~\ref{alg:hats} (right).}
\label{fig:eigComp3D}
\end{figure}
%
%
%%%%%%%%%%%%%%%%%%%%%%%%%%%%%%
\subsection{Nonlinear Gross-Pitaevskii eigenvalue problem}
In this final example we consider the nonlinear version of the Schr\"odinger equation, describing quantum-physical processes with interaction. This GPEVP has the form 
\begin{align*}
%\label{eq:GPE:strong}
  -\Delta u + Vu + \delta|u|^2 u= \lambda u
\end{align*}
with~$\delta\ge 0$ regulating the nonlinearity. For moderate values of~$\delta$ similar localization results for the groundstates can be observed~\cite{AltPV18}. As a result, the here presented localization preserving preconditioner promises good approximation results. We emphasize that this nonlinear case is not easily treated by the landscape function approach. Considering again a finite element discretization as in Section~\ref{sec:evp:fem}, we obtain the nonlinear matrix eigenvalue problem 
\[
  A_h u_h + \delta N_h(u_h)u_h = \lambda M_h u_h,  
\]
where~$N_h(u_h)$ equals the finite-dimensional approximation of the nonlinearity. Inspired by the inverse iteration for the linear case, i.e., $u^{n+1}_h = A_h^{-1} M_h u^n_h$, one may consider the iteration 
\[
  u^{n+1}_h = (A_h + \delta N_h(u^n_h))^{-1} M_h u^n_h 
\]
with an additional normalization step. This then leads to the results shown in Figure~\ref{fig:nonlinear}, illustrating that the groundstate can be well approximated by local functions also in the nonlinear case. It goes without saying that the iteration scheme may be replaced by more advanced methods, e.g., based on gradient flows~\cite{BaoD04, HenP18ppt,AltHP19ppt}. 
%
% nonlinear 2D random potential and eigenfunctions
% parameters: eps=7, h=8, beta=5*eps^{-2}, rng(145), random potential
% different values of delta, u0 = Mh*1, 100 INVIT steps
\begin{figure}
	\includegraphics[width=4.8cm, height=4.8cm]{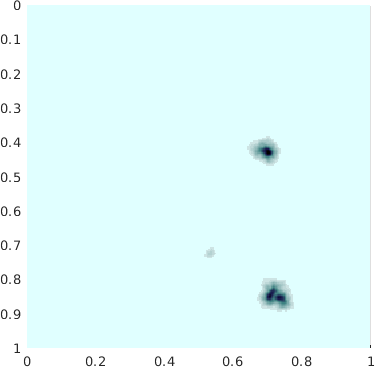}
	\includegraphics[width=4.8cm, height=4.8cm]{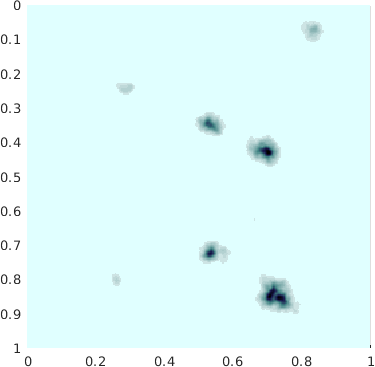}
	\includegraphics[width=4.8cm, height=4.8cm]{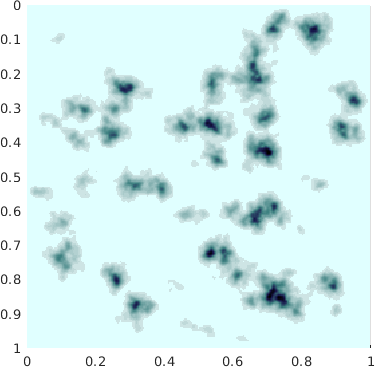}
\caption{Approximation of the groundstate for the Gross-Pitaevskii eigenvalue problem for~$\delta=1$ (left), $\delta=10$ (middle), and~$\delta=100$ (right) for a two-dimensional random potential with~$\ell_\eps=7$, $\beta = 5\cdot\eps^{-2}$. Results after 100 iteration steps with initial function being constant 1.} % plus 0-Rand !!
\label{fig:nonlinear}
\end{figure}
%
%
%
%=============================================================================
%=========  Conclusion
%=============================================================================
\section{Conclusion}\label{sec:conclusion}
In this paper, we have constructed a novel iterative scheme using analytical inside and adapting numerical analysis techniques used in localization proofs of the eigenstates \cite{AltHP18ppt}. The novel method is solely based on local operations (relative to the oscillation/correlation length of the potential) and, hence, able to approximate eigenfunctions of the random Schr\"odinger operator up to high precision. For this, we exploit the fact that oscillatory, high-amplitude random potentials lead to a localization of the lowermost eigenstates, which allows an efficient computation. 

The numerical experiments prove the applicability of the method also for the three-dimensional case as well as the corresponding nonlinear eigenvalue problem. Future research aims to further develop this approach in view of other applications with similar localization effects (such as light in a disordered medium~\cite{NatureLight97}) as well as time-dependent problems. 
%
%
%=============================================================================
%=========  Bibliogrphy
%=============================================================================
\newcommand{\etalchar}[1]{$^{#1}$}

\end{document}